\documentclass[10pt,noindent]{article}
\usepackage{latexsym,amsmath,amsfonts,hyperref}
\usepackage{graphics,color}

\usepackage{bbm}

\newtheorem{theorem}{Theorem}
\newtheorem{lemma}[theorem]{Lemma}

\newtheorem{Assumption}{Assumption}
\newtheorem{remark}{Remark}

\newcommand{\thmref}[1]{Theorem~\ref{thm:#1}} 
\newcommand{\lemref}[1]{Lemma~\ref{lem:#1}} 

\newcommand{\telque}{\, \mbox{ s.t. } \,} 

\newcommand{\R}{\mathbb{R}} 
\newcommand{\N}{\mathbb{N}} 

\newcommand{\Bcal}{\mathcal{B}}

\newcommand{\Xcal}{\mathcal{X}}

\newcommand{\s}{\sigma}

\newcommand{\paren}[1]{\left( \left. #1 \right. \right)} 
\newcommand{\cro}[1]{\left[ \left. #1 \right. \right]} 
\newcommand{\set}[1]{\left\{ \left. #1 \right. \right\}}
\newcommand{\absj}[1]{\left\lvert #1 \right\rvert} 

\newcommand{\Ind}[1]{{\bf 1}_{#1}}

\renewcommand{\P}{\mathbb{P}}
\renewcommand{\Pr}[1]{\P\paren{#1}}
\newcommand{\Prp}[2]{\P_{#1}\paren{#2}}
\DeclareMathOperator{\supp}{supp} 
\newcommand{\E}{\mathbb{E}} 
\newcommand{\Ex}[1]{\E\cro{#1}}
\newcommand{\Exp}[2]{\E_{#1}\cro{#2}}
\newcommand{\Var}{\mathrm{Var}}

\newcommand{\hyptag}[1]{\tag{\ensuremath{\mathbf{#1}}}} 

\newcommand\QED{\ifhmode\allowbreak\else\nobreak\fi
\quad\nobreak$\Box$\medbreak}
\newcommand{\proofstart}{\par\noindent\sl Proof:\rm\enspace}
\newcommand{\proofend}{\QED\par}
\newenvironment{proof}{\proofstart}{\proofend}

\newcommand{\Sc}{{S}}

\newtheorem{postita}{Post-it}

\begin{document}


\title{The number of potential winners in Bradley-Terry model in random environment}
\author{Raphael Chetrite \and Roland Diel \and Matthieu Lerasle}
\maketitle
%



 




%


\begin{abstract} We consider a Bradley-Terry model in random environment where each player faces each other once. More precisely the strengths of the players are assumed to be random and we study the influence of their distributions on the asymptotic number of potential winners.
First we prove that under mild assumptions, mainly on their moments, if the strengths are unbounded, the asymptotic probability that the best player wins is 1. We also exhibit a sufficient convexity condition to obtain the same result when the strengths are bounded. When this last condition fails, the number of potential winners grows at a rate depending on the tail of the distribution of strengths. We also study the minimal strength required for an additional player to win in this last case. 
\end{abstract}


\noindent
{\small {{\bf Keywords and phrases :} Bradley-Terry model, random environment, paired comparison.} }


\section{Introduction and Results}

We consider here a model of paired comparisons which may be used as a proxy for sport competition, chess tournament or comparisons of medical treatments. A set of $N$ players (teams, treatments, \dots) called $\set{1,\ldots,N}$ face each other once by pairs with independent outcomes. When $i$ faces $j$, the result is described by a Bernoulli random variable $X_{i,j}$ that is equal to $1$ when $i$ beats $j$ and $0$ if $j$ beats $i$ (hence $X_{i,j}=1-X_{j,i}$). The final result is given by the score $S_i = \sum_{j\neq i}X_{i,j}$ of each player that is its number of victories. We call winner every player that ends up with the highest score.

To each player $i$ is assigned a positive 
random variable $V_i$ modeling its intrinsic value, that is its "strength" or its "merit". 
Given $\mathbb{V}_{1}^{N} =(V_1,\dots,V_N)$, the distribution of $(X_{i,j})_{1\le i<j\le N}$ follows the Bradley-Terry model: all matches are independent and 
\begin{equation}
\forall 1\le i<j\le N,\qquad\Pr{X_{i,j}=1|\mathbb{V}_{1}^{N}}=\frac{V_i}{V_i+V_j}\enspace.
\label{BT}
\end{equation}

The distribution of $\mathbb{V}_{1}^{N}$ 
is chosen as follows. Let $\mathbb{U}_1^N=(U_1,\dots,U_N)$ denote i.i.d. random variables; to avoid trivial issues, suppose that all $U_i$ are almost surely positive. For any $i\in\set{1,\ldots,N}$, $V_{i}$ denotes the $i$-th order statistic of the vector $\mathbb{U}_1^N$: the larger the index of a player is, the "stronger" he is.

 Bradley-Terry model has been introduced independently by Zermelo \cite{Zer:1929} and Bradley and Terry \cite{Bradley_Terry:1952}, it was later generalized to allow ties \cite{ Dav:1970,Rao_Kup:1967} or to incorporate within-pair order effects \cite{Dav_Bea:1977}, see \cite{Bra:1976} for a review. 
Despite its simplicity, it has been widely used in applications for example to model sport tournaments, reliability problems, ranking scientific journals (see \cite{Bra:1976} or more recently \cite{Cat:2012} for references). Bradley-Terry model has also been studied in statistical literature, see for example \cite{David:1963,HasTib:1998,Simons_Yao:1999,Hun:2004,GliJen:2005,YanYangXu:2011} and references therein. 

Nevertheless, Bradley-Terry model has rarely been associated to random environment models (see however \cite{Sir_Red:2009}) and, to the best of our knowledge, never from a strictly mathematical point of view.  The addition of a random environment seems however natural as it allows to manage the heterogeneity of  strengths of players globally, without having to look at each one specifically. It is a method already used fruitfully in other areas such as continuous or discrete random walks (see \cite{Zei:2012} or \cite{DreRam:2014} for recent presentations). Our problem here is to understand how the choice of the distribution for the strengths of players influences the ranking of the players. In particular, does a player with the highest strength ends up with the highest score? And if not, what proportion of players might win? These problematics are related to the articles detailed below.

\begin{itemize}
\item Ben-Naim and Hengartner \cite{Ben_Hen:2007} study the number of players which can win a competition. These authors consider a simple model where the probability of upset $p<\frac{1}{2}$ is independent of the strength of players: 
$$\forall 1\le i<j\le N,\qquad \Pr{X_{i,j}=1}=p1_{i<j}+\left(1-p\right)1_{i>j}\enspace.
\label{simple}$$ For this model, they heuristically show with scaling techniques coming from polymer physics (see \cite{Gen:1979}) that, for large $N$, the number of potential champions behaves as $\sqrt{N}$. In the Bradley-Terry model in random environment, \thmref{SizeSetWinners} shows that the class of possible behaviors for this set is much richer.

\item Simons and Yao \cite{Simons_Yao:1999}  estimate the merits $\mathbb{V}_{1}^{N} $ based on the observations of $(X_{i,j})_{1\le i<j\le N}$. They prove consistency and asymptotic normality for the maximum likelihood estimator. It is interesting to notice that this estimator sorts the players  in the same order as the scores $S_i$ (see \cite{For:1957}). In particular, the final winner is always the one with maximal estimated strength. \thmref{QuandAfails} shows that usually this player has also maximal strength when the merits are unbounded but it is not always true in the bounded case, see \thmref{SizeSetWinners}.

\end{itemize}

Throughout the article, $U$ denotes a copy of $U_1$ independent of $\mathbb{U}_1^N$, $Q$ denotes the tail distribution function and $\supp Q$ its support, $\P$ denotes the annealed probability of an event with respect to the randomness of $\mathbb{V}_{1}^{N}$ and $(X_{i,j})_{1\le i< j\le N}$, while $\P_V$ denotes the quenched probability measure given $\mathbb{V}_{1}^{N}$, that is $\Pr{\ \cdot\ |\mathbb{V}_{1}^{N}}$. In particular,
$$
\forall 1\le i<j\le N,\qquad \Prp{V}{X_{i,j}=1}=\frac{V_i}{V_i+V_j}\enspace.
$$
We are interested in the asymptotic probability that the "best" player wins, that is that the player $N$ with the largest strength $V_N$ ends up with the best score. The following annealed result gives conditions under which this probability is asymptotically $1$ when the number of players $N\to \infty$. 

\begin{theorem}\label{thm:QuandAfails}
 Assume that there exist $\beta\in(0,1/2)$ and $x_0>0$  in the interior of $\supp Q$ such that $Q^{1/2-\beta}$ is convex on $[x_0,\infty)$ and that $\Ex{U^2}<\infty$.
 Then,
\[\Pr{\text{the player $N$ wins}}\ge \Pr{S_N>\max_{1\leq i\leq N-1}S_i}\xrightarrow[N\to\infty]{}1\enspace.\]
\end{theorem}

\noindent

\begin{remark}
 When the support of the distribution of $U$ is $\R_+$, the convexity condition is not very restrictive as it is satisfied by standard continuous distributions with tails function $Q(x)\simeq e^{-x^a}$, $Q(x)\simeq x^{-b}$ or $Q(x)\simeq (\log x)^{-c}$. 
%
The moment condition $\Ex{U^2}<\infty$ is more restrictive in this context but still allows for natural distributions of the merits as exponential, exponential of Gaussian or positive parts of Gaussian ones. 
It provides control of the explosion of maximal strengths. 
It is likely that it can be improved, but it is a technical convenience allowing to avoid a lot of tedious computations.

 When $\supp Q$ is finite, we can always assume by homogeneity that it is included in $[0,1]$, since the distribution of $(X_{i,j})_{1\le i<j\le N}$ given $\mathbb{V}_{1}^{N}$ is not modified if all $V_i$ are multiplied by the same real number $\lambda$. The moment condition is always satisfied and the only condition is the convexity one. This last condition forbids an accumulation of good players with strength close to 1. 
\end{remark}

Let us investigate now the necessity of the convexity condition. For this purpose, suppose that $Q(1-u)\sim u^{\alpha}$ when $u\to 0$, then the convexity condition holds iff $\alpha >2$. To check the tightness of the bound $2$, we introduce the following condition.
\begin{Assumption}
The maximum of $\supp Q$ is $1$ and there exists $\alpha\in [0,2)$ such that, 
\begin{equation}\label{eq:AssumptionA}
\hyptag{A}\log Q(1-u)= \alpha\log(u) +o(\log u)\quad\text{when $u\to 0$}.
\end{equation}
\end{Assumption}
Let us stress here that $Q$ may satisfy \eqref{eq:AssumptionA} even if it is not continuous. Moreover, $\alpha$ is allowed to be equal to $0$, in particular, $Q(1)$ may be positive. Notice that some standard  distributions satisfy Assumption \eqref{eq:AssumptionA}, for example the uniform distribution satisfies \eqref{eq:AssumptionA} with $\alpha=1$, the Arcsine distribution satisfies it with $\alpha=1/2$ and any Beta distribution $B(a,b)$ satisfies it as long as the parameter $b<2$ with $\alpha =b$.

The next quenched result studies, under Assumption \eqref{eq:AssumptionA}, the size of the set of possible winners.
\begin{theorem}\label{thm:SizeSetWinners}
For any $r\in\R_+$, let $G_{r}=\set{\lceil N-r\rceil+1,\ldots, N}$ 
denote the set of the $\lfloor  r\rfloor$ best players.
If \eqref{eq:AssumptionA} holds, for any $\gamma<1-\alpha/2$ then, $\P$ almost-surely, 
\[\Prp{V}{\mbox{none of the }N^\gamma\mbox{ best players wins}}=\Prp{V}{\max_{i\in  G_{N^{\gamma}}}S_i< \max_{i\in G_N}S_i}\to 1\enspace.\]
For any $\gamma>1-\alpha/2$ then, $\P$ almost-surely, 
\[\Prp{V}{\mbox{one of the }N^\gamma\mbox{ best players wins}}=\Prp{V}{\max_{i\notin  G_{N^\gamma}}S_i< \max_{i\in G_N}S_i}\to 1\enspace.\]
\end{theorem}

\begin{remark} The first part of the theorem shows that, when $Q(1-u)\sim u^{\alpha}$, with $\alpha<2$, then none of $N^{\gamma}$ "best" players, for any $\gamma\in (0,1-\alpha/2)$ wins the competition. 
In particular, the "best" one does not either.  In this sense, the bound $2$ in the asymptotic development of $Q$ around $1$ is tight.  

%

The second result in \thmref{SizeSetWinners} shows the sharpness of the bound $1-\alpha/2$ in the first result.
Heuristically, this theorem shows that under Assumption \eqref{eq:AssumptionA}, $N^{1-\alpha/2}$ players  can be champion. 
\end{remark}

Under Assumption~\eqref{eq:AssumptionA}, the best player  does not win the championship. Therefore, we may wonder what strength $v_{N+1}$ an additional tagged player $N+1$ should have to win the competition against players distributed according to $Q$. The following quenched result discusses the asymptotic probability that player $N+1$  wins depending on its strength $v_{N+1}$. To maintain consistency with the previous results, we still use the notations $$S_i=\sum_{j=1,j\ne i}^NX_{i,j}\quad \text{ for }i\in\set{1,\ldots,N+1}\enspace.$$ With this convention, $S_{N+1}$ describes the score of the player $N+1$ and the score of each player $i\in\set{1,\ldots, N}$ is equal to $S_i+X_{i,N+1}$.

\begin{theorem}\label{thm:Cutoff}
Assume \eqref{eq:AssumptionA} and let 
$$\vartheta_U=\Ex{\frac{U}{(U+1)^2}}\quad\text{ and }\quad\epsilon_N=\sqrt{\frac{2-\alpha}{\vartheta_U}\frac{\log N}{N}}\enspace.$$ 

If $\liminf_{N\to \infty}\frac{v_{N+1}-1}{\epsilon_N}>1$, then, $\P$-almost surely
 \[\Prp{V}{\mbox{player }N+1\mbox{ wins}}\geq\Prp{V}{S_{N+1}>1+\max_{i=1,\ldots,N}S_i}\to 1\enspace.\]\\
 If $\limsup_{N\to \infty}\frac{v_{N+1}-1}{\epsilon_N}<1$, then, $\P$-almost surely
 \[\Prp{V}{\mbox{player }N+1\mbox{ does not win}}\geq\Prp{V}{S_{N+1}<\max_{i=1,\ldots,N}S_i}\to1\enspace.\]
\end{theorem}

\begin{remark}
This result shows a cut-off phenomenon around $1+\epsilon_N$ for the asymptotic probability that player $N+1$ wins. 

It is interesting to notice that, for a given $\alpha$, $\epsilon_N$ is a non increasing function of $\vartheta_U$. Therefore, when $U$ is stochastically dominated by $U'$, that is  $\P(U\geq a)\leq \P(U'\geq a)$ for any $a\in[0,1]$, 
we have $\vartheta_U\le \vartheta_{U'}$, hence 
 $\epsilon_N^U\ge \epsilon_N^{U'}$. In other words, it is easier for the tagged player to win against opponents distributed as $U'$ than as $U$ even if the latter has a weaker mean than the former. This result may seem counter-intuitive at first sight.  In the following example in particular, it is easier for the additional player to win the competition in case 1 than in case 2, since both distributions satisfy \eqref{eq:AssumptionA} with $\alpha=0$.
\begin{enumerate}
 \item All players in $\set{1,\ldots,N}$ have strength $1$.
 \item The players in $\set{1,\ldots,N}$ have strength $1$ with probability $1/2$ and strength $1/2$ with probability $1/2$.
\end{enumerate}
Actually the score of the tagged player is smaller when he faces stronger opponents as expected, but so is the best score of the other good players. 
\end{remark}

Remark that the first theorem is an annealed result while the others are quenched. Indeed,  the first theorem requires to control precisely the difference of strengths between the best player and the others when all the players are identically distributed, this seems complicated in the quenched case. This problem does not appear in the other results: for example, in Theorem 
\ref{thm:Cutoff}, the  strength of the tagged player is deterministic and the strengths of others are bounded by 1.

\medskip


The remaining of the paper presents the proofs of the main results. Section~\ref{sec:TheBestWinsUsually} gives the proof of \thmref{QuandAfails} and Section~\ref{sec:ButNotAlways} the one of Theorems~\ref{thm:SizeSetWinners} and \ref{thm:Cutoff}.

%
%
%


\section{Proof of \thmref{QuandAfails}}\label{sec:TheBestWinsUsually}
%
%
%
%
%

Denote by $Z_{N}=\max_{i\in\set{1,\ldots,N-1}}S_i$. The key to our approach is to build random bounds $s^N$ and $z^N$ depending only on $\mathbb{V}_{1}^{N}$ such that, 
\begin{equation}\label{eq:ConSZ1}
\Pr{S_{N}\ge s^N}\to 1,\quad \Pr{Z_N\le  z^N}\to 1\quad \text{and}\quad \Pr{s^N>z^N}\to 1\enspace.
\end{equation}
It follows that, 
\begin{multline*}
\Pr{S_{N}>Z_N}\ge \Pr{S_{N}\ge s^N,\; Z_N\le z^N,\;s^N>z^N}\\
\quad\ge 1-\Pr{S_{N}< s^N}-\Pr{ Z_N> z^N}-\Pr{s^N<z^N}\to 1\enspace. 
\end{multline*}
The construction of $s^N$ and $z^N$ is the subject of the next subsection, it is obtained thanks to concentration inequalities. The concentration of $S_{N}$ is easy, the tricky part is to build $z^N$. First, we use the bounded difference inequality to concentrate $Z_N$ around its expectation. The upper bound on its expectation is given by the sum of the expected score of player $N-1$ and a deviation term that is controlled based on an argument used by Pisier \cite{Pisier:1983}. Finally, the control of $\Pr{s^N>z^N}$ derives from an analysis of the asymptotics of $V_{N-1}$ and $V_N$.

  \subsection{Construction of $s^N$ and $z^N$}
 The expectation of the score $S_N$ of the best player is given by
 \[\Exp{V}{S_N}=\sum_{i=1}^{N-1}\frac{V_{N}}{V_{N}+V_i}\]
 and the concentration of $S_N$ is given by Hoeffding's inequality, see \cite{Hoeffding:1963}:
\[
	\Prp{V}{S_N\le \sum_{i=1}^{N-1}\frac{V_{N}}{V_{N}+V_i}-\sqrt{\frac{Nu}2}}\le e^{-u}\enspace.
\]
Hence, the first part of \eqref{eq:ConSZ1} holds for any $u_N\to\infty$ with
\begin{equation}\label{borneinfS}
s^N=\sum_{i=1}^{N-1}\frac{V_{N}}{V_{N}+V_i}-\sqrt{Nu_N}\enspace. 
\end{equation}
The following lemma implies the concentration of $Z_N$ around its expectation. As we will use it in a different context in other proofs, we give a slightly more general result.
\begin{lemma}\label{lem:ConcMax}
Let $I\subset [N]$ and let $Z=\max_{i\in I}S_i$. For any $u>0$,
 \begin{equation*}
 \Prp{V}{Z\geq \Exp{V}{Z}+\sqrt{\frac{N}2u}}\leq e^{-u}
 \end{equation*}
 and
\begin{equation*}
 \Prp{V}{Z\leq \Exp{V}{Z}-\sqrt{\frac{N}2u}}\leq e^{-u}\enspace.
 \end{equation*}
\end{lemma}
\begin{proof}
The proof is based on the bounded difference inequality recalled in \thmref{bdi} of the appendix (see \cite{Dembo:1997}, \cite{Marton:1996} or \cite{Massart:2007}). To apply this result, we have to decompose properly the set of independent random variables $(X_{i,j})_{1\le i<j\le N}$. To do so, we use the round-robin algorithm which we briefly recall.

First, suppose $N$ even. Denote by $\s$ the permutation on $\{1,\dots,N\}$ such that $\s(1)=1,\ \s(N)=2\ \text{and}\ \s(i)=i+1,\ \text{if } 1< i< N$ and define the application 
  $$A: \{1,\dots,N-1\}\times \{1,\dots,N\}\to \{1,\dots,N\}$$
  by $A(k,i)=\s^{-(k-1)}(N+1-\s^{(k-1)}(i))$. Then, for any $k\in\{1,\dots,N-1\}$, $A(k,\cdot)$ is an involution with no fixed point and for any $i\in\{1,\dots,N\}$, $A(\cdot,i)$ is a bijection from  $\{1,\dots,N-1\}$ to $\{1,\dots,N\}\setminus\{i\}$. The first variable of function $A$ has to be thought as ``steps'' in the tournament. At each step, every competitor plays exactly one match and $A(k,i)$ represents here the opponent of player $i$ during the $k^\text{th}$ step. We denote by $Z^k$ the variables describing the results of the $k^\text{th}$ step, that is $Z^k=(X_{i,A(k,i)},\ i<A(k,i))$. The variable $Z$ can be expressed as a (measurable) function of the $Z^k$, $Z=\Psi(Z^1,\ldots,Z^{N-1})$. Moreover, for any $k=1,\ldots,N-1$ and any $z^1,\ldots,z^{N-1},\tilde{z}^{k}$ in $\set{0,1}^{N/2}$, the differences
 \[\absj{\Psi(z^1,\ldots,z^{k},\ldots,z^{N-1})-\Psi(z^1,\ldots,\tilde{z}^{k},\ldots,z^{N-1})}\]
are bounded by $1$. If $N$ is odd, we only have to add a ghost player and $Z$ can be  expressed in the same way as a measurable function of $N$ independent random variables with differences bounded by $1$.
Therefore, in both cases, the bounded difference inequality applies and gives the result.
\end{proof}
It remains to compare the expectations of $Z_N$ and of $S_N$. This requires to control the sizes of $V_{N-1}$  and $V_N$. Recall that $V_i$ is the $i^{\text{th}}$ order statistics of the vector $\mathbb{U}_1^N=\left(U_1,\ldots,U_N\right)$, that is
$$V_i=\min_{k\in\{1,\dots,N\}}\set{U_k\ \big|\ \exists I\subset \{1,\dots,N\},\ |I|=i \text{ and } \forall l\in I, U_l\leq U_k}\enspace .$$
Then the sets $\set{V_1,\ldots,V_N}$ and $\set{U_1,\ldots,U_N}$, counted with multiplicity, are equal which guarantees that, for any function $f$, $\sum_{i=1}^Nf(V_i)=\sum_{i=1}^Nf(U_i)$.

Let $Q^{-1}$ denote the generalized inverse of $Q$: for $y\in(0,1)$,
$$
Q^{-1}(y)=\inf\set{x\in\R_+^*,\ Q(x)\leq y}\enspace.
$$
Remark that the convexity assumption implies that, if $M$ is the supremum of $\supp Q$, the function $Q$ is a continuous bijection from $[x_0,M)$ to $(0,Q(x_0)]$ such that $\lim_{x\to M}Q(x)=0$ so,  on $(0,Q(x_0)]$, $Q^{-1}$ is the true inverse of $Q$. 
\begin{lemma}\label{lem:bornesvn}
For any function $h$ defined on $\R_+$ such that $\lim_{+\infty}h=+\infty$, let
\begin{align*}
 a^h_N&=\left\{
\begin{array}{ll}
 Q^{-1}\paren{\frac{h(N)}N}&\;\mbox{if}\;\frac{h(N)}N<1\\
 0&\;\mbox{ otherwise}
\end{array}\right.\\
 b^h_N&=Q^{-1}\paren{\frac{1}{N h(N)}}\\
 A^h_N&=\set{a_N^h\leq V_{N-1}\leq V_{N}\leq b^h_N}\enspace. 
\end{align*}
Then $\displaystyle\lim_{N\to\infty}\Pr{A^h_N}=1$.
\end{lemma}
\begin{proof}
If $h(N)/N\ge 1$, $\Pr{V_{N-1}< a^h_N}=0$ so the lower bound is trivial. If $h(N)/N<1$ and $h(N)\ge 1$, since $ x\wedge Q(x_0)\le  Q(Q^{-1}(x))\le x$,
\begin{align*}
\Pr{V_{N-1}< a^h_N}&=(1-Q( a^h_N))^N+N(1-Q( a^h_N))^{N-1}Q( a^h_N)\\
&\leq 2h(N)\paren{1-\frac{h(N)}N\wedge Q(x_0)}^{N-1}\enspace.
\end{align*}
Hence, $\Pr{V_{N-1}< a^h_N}\to 0$. Moreover, for any $N$ such that $1/(Nh(N))\le x_0$,
\begin{align*}
\Pr{V_{N}> b^h_N}&=1-\Pr{U< b^h_N}^N=1-\paren{1-\frac{1}{N h(N)}}^N\to 0\enspace.
\end{align*}
\end{proof}

\begin{lemma} \label{lem:cvgv} 
There exists a non-increasing deterministic function $y\to\eta(y)$ on $\R_+$ such that $\lim_{+\infty}\eta=0$ and $
\lim_{N\to\infty}\Pr{B_N}=1
$,
where $B_N=\set{\frac{V_{N}}{\sqrt{N}}\leq \eta(N)}$.
\end{lemma}

\begin{proof}
Since $\Ex{U^2}=\int_0^\infty yQ(y) d y<\infty$, $\lim_{x\to+\infty}\int_x^\infty yQ(y) d y=0$. As $Q$ is non increasing,
	\begin{align*}
		\frac{3}{8}x^2Q(x)=\int_{x/2}^x yd y Q(x)\leq \int_{x/2}^x yQ(y)d y \to 0\enspace.
	\end{align*}
Therefore $Q(x)=o(1/x^2)$ when $x\to+\infty$. This implies that $Q^{-1}(y)=o(1/\sqrt{y})$ when $y\to0$ and there is a non-decreasing function $y\to u(y)$ defined on $\R_+$ such that $\lim_0u=0$ and $Q^{-1}(y)\leq u(y)/\sqrt{y}$. For any $N$ large enough, choosing $y=\frac{\sqrt{u(1/N)}}{N}$
$$Q^{-1}\left(\frac{\sqrt{u(1/N)}}{N}\right)\leq \sqrt{Nu(1/N)}.$$
Setting $\eta(y)=\sqrt{u(1/y)}$, Lemma \ref{lem:bornesvn} used with $h(x)=1/\sqrt{u(1/x)}$ gives the result.
\end{proof}

We also need the following result:

\begin{lemma}\label{lem:ContENV}
	Define $\displaystyle E_N(V)=\frac{1}{N}\sum_{i=1}^{N-1}\frac{V_iV_{N-1}}{V_{N-1}+V_i}$ and 
$$C_N=\set{\frac{1}{4}\Ex{U}\leq E_N(V)\leq 2\Ex{U}}\enspace.$$
Then, 
$\lim_{N\to\infty}\Pr{C_N}=1$.
\end{lemma}
\begin{proof}
Remark that for $i\in\set{1,\dots,N-1}$, $1\leq1+\frac{V_{i}}{V_{N-1}}\leq 2$, hence
$$
E_N(V)\geq  \frac{1}{2N}\sum_{i=1}^{N-1}V_i= \frac{1}{2N}\sum_{i=1}^{N}U_i-\frac{V_N}{2N}
$$
and
$$
 E_N(V)\leq \frac{1}{N}\sum_{i=1}^{N-1}V_i\le \frac{1}{N}\sum_{i=1}^{N}U_i\enspace .
$$
\lemref{cvgv} ensures that $V_N/N$ converges in probability to 0 and therefore the proof is easily conclude using the Weak Law of Large Numbers.
\end{proof}

We are now in position to bound $\Exp{V}{Z_{N}}$.

\begin{lemma}\label{lem:ENV}
For $N$ large enough, on $B_N\cap C_N$,
	\begin{align*}
 \Exp{V}{Z_{N}}&\le \sum_{i=1}^{N-1}\frac{V_{N-1}}{V_{N-1}+V_i}+\sqrt{8\Ex{U}\frac {N\log N}{V_{N-1}}}\enspace.
 \end{align*}
\end{lemma}

\begin{proof}
Write 
 \[Z_{N}'=Z_N-\sum_{i=1}^{N-1}\frac{V_{N-1}}{V_{N-1}+V_i}\enspace.\]
For any $\lambda>0$, Jensen's inequality and the argument of Pisier \cite{Pisier:1983} give
 \begin{align*}
 \Exp{V}{Z_{N}'}&\le \frac1{\lambda}\log\Exp{V}{e^{\lambda Z_{N}'}}\\
 &\le  \frac1{\lambda}\log\paren{\sum_{i=1}^{N-1}\Exp{V}{e^{\lambda\paren{S_i-\sum_{i=1}^{N-1}\frac{V_{N-1}}{V_{N-1}+V_i}}}}}\enspace.
 \end{align*}
 Let $\Sc=\sum_{i=1}^{N-1}X_i$, where $X_i$ are independent Bernoulli variables with respective parameters $V_{N-1}/(V_{N-1}+V_i)$. Every $S_i$ is stochastically dominated by $\Sc$, thus
 \begin{align*}
 \Exp{V}{Z_{N}'}&\le  \frac1{\lambda}\log\paren{N\Exp{V}{e^{\lambda\paren{\Sc-\sum_{i=1}^{N-1}\frac{V_{N-1}}{V_{N-1}+V_i}}}}}\enspace.
 \end{align*}
 Let 
\[\lambda_N=\sqrt{\frac{V_{N-1}\log N}{2N\Ex{U}}}\enspace,\]
for $N$ large enough, on $B_N\cap C_N$, $\lambda_N\le 3/(8e^2)<1$.
 \lemref{LapBernoulli} in Appendix evaluates the Laplace transform of Bernoulli distribution and gives here
\begin{align}
\notag\Exp{V}{e^{\lambda_N\paren{\Sc-\sum_{i=1}^{N-1}\frac{V_{N-1}}{V_{N-1}+V_i}}}}&\le e^{\lambda_N^2\sum_{i=1}^{N-1}\frac{V_{N-1}V_i}{(V_{N-1}+V_i)^2}}\enspace.
\end{align}
Therefore,
\begin{align*}
 \Exp{V}{Z_{N}'}&\le \frac{\log N}{\lambda_N}+\frac{2\lambda_NN}{V_{N-1}}\Ex{U}= \sqrt{8\Ex{U}\frac {N\log N}{V_{N-1}}}\enspace.
 \end{align*}	
\end{proof}
Lemmas~\ref{lem:ConcMax}, \ref{lem:cvgv}, \ref{lem:ContENV} and \ref{lem:ENV} give the second part of
\eqref{eq:ConSZ1} for any $u_N\to\infty$ with 
 \begin{equation}\label{bornesupZ}
 z^N=\sum_{i=1}^{N-1}\frac{V_{N-1}}{V_{N-1}+V_i}+\sqrt{8\Ex{U}\frac {N\log N}{V_{N-1}}}+\sqrt{Nu_N}\enspace. 
 \end{equation}

To prove the third item of \eqref{eq:ConSZ1}, it remains to prove that there exists some $u_N\to\infty$ such that the probability of the following event tends to one:
\begin{align}\label{eq:aobtenir}
 \sum_{i=1}^{N-1}\paren{\frac{V_{N}}{V_i+V_{N}}-\frac{V_{N-1}}{V_i+V_{N-1}}}>\sqrt{8\Ex{U}\frac {N\log N}{V_{N-1}}}+2\sqrt{Nu_N}\enspace.
\end{align}

\subsubsection{Proof of \eqref{eq:aobtenir}}
The proof relies on a precise estimate of the difference $V_N-V_{N-1}$. 
\begin{lemma}\label{lem:ControlIcalN}
Let
\[D_N=\set{ V_{N-1}\le V_{N}\left(1-\left(\frac{V_{N-1}}{\sqrt{N}}\right)^{1-\beta}\right)}\enspace.\]
Then $\Pr{D_N}\to 1$ as $N\to \infty$.
\end{lemma}

\begin{proof}
In the proof, $C_\beta$ denotes a deterministic function of $\beta$ which value can change from line to line and $F=1-Q$ denotes the c.d.f. of $U$.

Let $\eta$ be the function defined in \lemref{cvgv} and denote $h=\eta^{-\beta}$. As $x_0$ is in the interior of the support of $Q$, there exists a constant $c<1$ such that $x_0/c$ also lies in the interior of $\supp Q$. Let $a_N^h,b_N^h$ be defined as in \lemref{bornesvn} and consider the event 
\[F_N=\set{\left(x_0/c\right)\vee a_N^h\leq V_{N-1}\leq \sqrt{N}(1/2\wedge \eta(N))\wedge b_N^h}\enspace.\]
 According to \lemref{bornesvn} and \lemref{cvgv}, $\lim_{N\to\infty}\Pr{F_N}=1$. Moreover,
\begin{align*}
 \Pr{D_N^c}&\le \Pr{D_N^c\cap F_N}+\Pr{F_N^c}\\
 &=\Ex{\Ind{F_N}\Pr{D_N^c|V_{N-1}}}+\Pr{F_N^c}\enspace.
 \end{align*}
The cumulative distribution function  of the random variable $V_{N}$ given $V_{N-1}$ is $1-Q/Q(V_{N-1})$ and then
\[\Pr{D_N^c|V_{N-1}}= \frac{Q(V_{N-1})-Q\paren{V_{N-1}\paren{1-\paren{V_{N-1}/\sqrt{N}}^{1-\beta}}^{-1}}}{Q(V_{N-1})}\enspace.\]
On the event $F_N$, the convexity of $Q$ gives 
 \begin{align*}
\Pr{D_N^c|V_{N-1}}&\leq \frac{\paren{V_{N-1}/\sqrt{N}}^{1-\beta}}{1-\paren{V_{N-1}/\sqrt{N}}^{1-\beta}}\frac{V_{N-1}F'(V_{N-1})}{Q(V_{N-1})}
\end{align*}
  and $V_{N}/\sqrt{N}$ is smaller than $1/2$, thus
\begin{align}\label{majBN}
\Pr{D_N^c|V_{N-1}}&\leq C_\beta\frac{V_{N-1}^{2-\beta}F'(V_{N-1})}{N^{(1-\beta)/2}Q(V_{N-1})}\enspace.
\end{align}
Moreover, by convexity of $Q^{1/2-\beta}$, the function $F'/Q^{1/2+\beta}$ is non-increasing, hence, by H\"older's inequality, for any $x\geq x_0/c$,
\begin{align*}
x^{2-2\beta}\frac{F'(x)}{Q(x)^{1/2+\beta}}&\le C_\beta\int_{cx}^x\frac{y^{1-2\beta}F'(y)}{Q(y)^{1/2+\beta}}dy\\
&\le C_\beta\paren{\int_0^{\infty}y^2F'(y)dy}^{1/2-\beta}\paren{\int_{cx}^{x}\frac{F'(y)}{Q(y)}dy}^{1/2+\beta}\\
&\le C_\beta\paren{\log\paren{\frac{Q(cx)}{Q(x)}}}^{1/2+\beta}.
\end{align*}

As seen in the proof of Lemma \ref{lem:cvgv}, $\lim_{y\to\infty}y^2Q(y)=0$, hence $Q(cx)\leq C_{\beta,1}/x^2$ for some constant $C_{\beta,1}$. Therefore for any $x\ge x_0/c$,
\begin{align*}
	x^{2-2\beta}\frac{F'(x)}{Q(x)^{1/2+\beta}}\le C_{\beta}\paren{\log\paren{\frac{C_{\beta,1}}{x^2Q(x)}}}^{1/2+\beta}\enspace.
\end{align*}
The function $g(x)=(x^2Q(x))^{\beta/4}\paren{\log\frac{C_{\beta,1}}{x^2Q(x)}}^{1/2+\beta}$ is upper bounded, this   yields the following inequality:
\begin{align*}
	x^{2-3\beta/2}F'(x)\le C_\beta Q(x)^{1/2+3\beta/4}\enspace.
\end{align*}
This last bound applied to $x=V_{N-1}$  combined with \eqref{majBN} leads to:
\begin{align*}
	\Pr{D_N^c|V_{N-1}}&\leq \frac{C_\beta\left(V_{N-1}\right)^{\beta/2}}{N^{(1-\beta)/2}\paren{Q(V_{N-1})}^{1/2-3\beta/4}} \quad\text{on }F_N.
\end{align*}
Now on $F_N$ the following bounds also hold: 
$$
V_{N-1}\leq \sqrt{N}\eta(N)\text{ and } Q(V_{N-1})\geq \frac{\eta(N)^{\beta}}{N}\enspace,$$
hence,
\begin{align*}
	\Ind{F_N}\Pr{D_N^c|V_{N-1}}&\leq C_{\beta}\eta(N)^{3\beta^2/4}\enspace.
\end{align*}
We conclude the proof by integration of the last inequality.
\end{proof}

Let $h(x)=x/2$ and consider the event 
\begin{align*}
 	G_N&=A_N^{h}\cap B_N\cap C_N\cap D_N\\
 	&=\left\{V_{N-1}\ge Q^{-1}(1/2),\;\frac{\sqrt{N}}{V_N}\geq\frac1{\eta(N)},\ V_{N}-V_{N-1}\geq V_{N}\paren{\frac{V_{N-1}}{\sqrt{N}}}^{1-\beta},\right.\\
	&\qquad  \left.\frac{\Ex{U}}4\le E_N(V)\le 2\Ex{U}\right\}.
 \end{align*}
 According to Lemmas~\ref{lem:bornesvn}, \ref{lem:cvgv}, \ref{lem:ContENV} and \ref{lem:ControlIcalN}, $\Pr{G_N}$ converges to $1$ when $N\to\infty$ so we only have to prove that \eqref{eq:aobtenir} holds on the event $G_N$. And on $G_N$, 
 \begin{multline*}
\sum_{i=1}^{N-1}\paren{\frac{V_{N}}{V_i+V_{N}}-\frac{V_{N-1}}{V_i+V_{N-1}}}\geq\frac{V_{N}-V_{N-1}}{2V_{N}} \sum_{i=1}^{N-1}\frac{V_i}{V_{N-1}+V_i}\\
 \ge
 \sqrt{N}\paren{\frac{\sqrt{N}}{V_{N-1}}}^{\beta}\frac{E_N(V)}{2}
 \ge \frac{\Ex{U}}{8}\sqrt{N}\paren{\frac{\sqrt{N}}{V_{N-1}}}^{\beta}.
 \end{multline*}
Thus, for $N$ large enough, on $G_N$,
\begin{align}
\frac{\sum_{i=1}^{N-1}\paren{\frac{V_{N}}{V_i+V_{N}}-\frac{V_{N-1}}{V_i+V_{N-1}}}}{\sqrt{8\Ex{U}\frac{N\log N}{V_{N-1}}}}&\ge \sqrt{\frac{\Ex{U}}{8^3}}\left(V_{N-1}\right)^{1/2-\beta}\frac{N^{\beta/2}}{\sqrt{\log N}}\notag\\
&\ge \sqrt{\frac{\Ex{U}}{8^3}}\left(Q^{-1}(1/2)\right)^{1/2-\beta}\frac{N^{\beta/2}}{\sqrt{\log N}}\ge 2
\label{cinqa}
\end{align}
and
\[\sum_{i=1}^{N-1}\paren{\frac{V_{N}}{V_i+V_{N}}-\frac{V_{N-1}}{V_i+V_{N-1}}}\ge \frac{\Ex{U}}{8}\sqrt{N}\paren{\frac{\sqrt{N}}{V_{N-1}}}^{\beta}\ge \frac{\Ex{U}}{8}\sqrt{N}\frac{1}{\eta(N)^{\beta}}\enspace.\]
Hence, for a constant $c$ small enough and $u_N=c/\paren{\eta(N)}^{2\beta}$, on $G_N$,
\begin{align}
2\sqrt{Nu_N}< \frac12\sum_{i=1}^{N-1}\paren{\frac{V_{N}}{V_i+V_{N}}-\frac{V_{N-1}}{V_i+V_{N-1}}}. \label{cinqb}
\end{align}
 Bounds \eqref{cinqa} and \eqref{cinqb} imply \eqref{eq:aobtenir}; this concludes the proof of \thmref{QuandAfails}.

\section{ Proof of \thmref{SizeSetWinners} and \thmref{Cutoff}}\label{sec:ButNotAlways}

Remark that in both theorems, each variable $S_i$ for $i\in \set{1,\ldots, N}$ has the same definition, it corresponds to the score of a player with strength $V_i$ playing against opponents with respective strength $\{V_j, j\in\set{1,\ldots, N}\setminus \set{i}\}$. Therefore in both proofs, the notation $Z_N=\max_{i\in\set{1,\ldots,N}}S_i$ represents the same quantity.

We build bounds $s^N_-<s^N_+$ and $z_-^N<z_+^N$ depending only on $\mathbb{V}_{1}^{N}$ such that
\begin{equation}\label{eq:ConSZ}
\Prp{V}{s^N_-\le S_{N+1}\le s^N_+}\to 1,\ \Prp{V}{z_-^N\le Z_N\le  z_+^N}\to 1, \quad \P-a.s. 
\end{equation}
and such that, when $\liminf_{N\to \infty}\frac{v_{N+1}-1}{\epsilon_N}>1$, $\P$-almost surely, for any $N$ large enough, $s^N_->1+z_+^N$, while when $\limsup_{N\to \infty}\frac{v_{N+1}-1}{\epsilon_N}<1$, $\P$-almost surely, for any $N$ large enough, $s^N_+<z_-^N$.
In the first case, it follows that, on $\set{s^N_->1+z_+^N}$, 
\begin{multline*}
\Prp{V}{S_{N+1}>1+Z_N}\ge \Prp{V}{S_{N+1}\ge s^N_-,\; Z_N\le z_+^N}\\
\quad\ge 1-\Prp{V}{S_{N+1}< s^N_-}-\Prp{V}{ Z_N> z_+^N}\enspace. 
\end{multline*}
The result in the second case is obtained with a similar argument. This will establish \thmref{Cutoff}.

For \thmref{SizeSetWinners}, given $\gamma_0<1-\alpha/2<\gamma_1$, we build random bounds $z^{N}_0$ and $z^{N}_1$ depending only on $\mathbb{V}_{1}^{N}$ such that $\P$-almost surely, 
\begin{equation}\label{eq:UpperBoundsinG}
\Prp{V}{\max_{i\in  G_{N^{\gamma_0}}}S_i\le z^{N}_0 }\to 1, \qquad \Pr{\max_{i\notin  G_{N^{\gamma_1}}}S_i\le z^{N}_1 }\to 1\enspace. 
\end{equation}

and
\[\Pr{\liminf\set{z^{N}_0<z_-^{N} }}= \Pr{\liminf\set{z^{N}_1<z_-^{N}} }=1\enspace.\]
On $\set{z^{N}_0<z_-^{N}}$,
\begin{multline*}
\Prp{V}{\max_{i\in  G_{N^{\gamma_0}}}S_i< Z_N}\ge \Prp{V}{\max_{i\in  G_{N^{\gamma_0}}}S_i< z^{N}_0, \; z_-^N< Z_N}\\
\ge1-\Prp{V}{\max_{i\in  G_{N^{\gamma_0}}}S_i\ge  z^{N}_0}-\Prp{V}{ z_-^N< Z_N}\enspace.
\end{multline*}
On $\set{z^{N}_1<z_-^{N} }$,
\begin{multline*}
\Prp{V}{\max_{i\notin  G_{N^{\gamma_1}}}S_i< Z_N}\ge \Prp{V}{\max_{i\notin  G_{N^{\gamma_1}}}S_i< z^{N}_1, \; z_-^N< Z_N}\\
\ge1-\Prp{V}{\max_{i\notin  G_{N^{\gamma_1}}}S_i\ge  z^{N}_1}-\Prp{V}{ z_-^N< Z_N}\enspace.
\end{multline*}
Together, these inequalities yield directly \thmref{SizeSetWinners}.

The construction of $s^N_-$ and $s^N_+$ will derive from the concentration of $S_{N+1}$ given by Hoeffding's inequality: for any $u>0$,
\begin{align}
\label{eq:lowerboundS}\Prp{V}{S_{N+1}\le \sum_{i=1}^N\frac{v_{N+1}}{V_i+v_{N+1}}-\sqrt{\frac{Nu}2}}&\le e^{-u}\enspace.\\
\label{eq:upperboundS}\Prp{V}{S_{N+1}\ge \sum_{i=1}^N\frac{v_{N+1}}{V_i+v_{N+1}}+\sqrt{\frac{Nu}2}}&\le e^{-u}\enspace. 
\end{align}
We will now build the bounds $z_0^{N}$, $z_1^{N}$, $z_-^N$ and $z_+^N$. To do so, we study the concentration of $Z_N$, $\max_{i\in  G_{k_N}}S_i$ and $\max_{i\notin  G_{\ell_N}}S_i$. 
The construction of these bounds is based on the same kind of arguments as the ones used in the previous section. The construction of $z_-^N$ requires a lower bound on $\Exp{V}{Z_N}$ which is obtained by comparison with the maximum of copies of the $S_i$ that are independent, see \lemref{tech}.

\subsection{Construction of $z_0^{N}$, $z_1^{N}$, $z_-^N$ and $z_+^N$}


\lemref{ConcMax} gives the concentration of $Z_N$, $\max_{i\in  G_{N^{\gamma_0}}}S_i$ and $\max_{i\notin  G_{N^{\gamma_1}}}S_i$ around their respective expectations 
which are evaluated in the following lemma.

\begin{lemma}\label{lem:AsymptoticOfEZ_N} $\P$ almost-surely,
\begin{align*}
&\Exp{V}{Z_N}= N\Ex{\frac{1}{1+U}}+\sqrt{(2-\alpha)\vartheta_UN \log N}+o(\sqrt{N \log N})\enspace,\\
& \Exp{V}{\max_{i\notin  G_{N^{\gamma_1}}}S_i}\le N\Ex{\frac1{1+U}}-N^{1/2+\nu}\vartheta_U+o\paren{N^{1/2+\nu}}\enspace,
\end{align*}
where $\nu=\frac{\gamma_1-(1-\alpha/2)}{2\alpha}>0$. In addition, $\P$ almost-surely,
\begin{multline*}
\Exp{V}{\max_{i\in  G_{N^{\gamma_0}}}S_i}\le N\Ex{\frac{1}{1+U}}+\sqrt{2\gamma_0\vartheta_UN \log N}+o(\sqrt{N \log N})\enspace.
\end{multline*}

\end{lemma}

\begin{proof}$ $\\
\emph{Upper bounds.}
Define  $Z_{N}'=Z_N-\sum_{k=1}^{N}\frac{1}{1+V_k}\enspace.$
The law of iterated logarithm ensures that, $\P$-almost surely,
\begin{equation}\label{eq:LogItere}
\sum_{k=1}^{N}\frac{1}{1+V_k}=\sum_{k=1}^N\frac{1}{1+U_k}=N\Ex{\frac{1}{1+U}}+o\paren{\sqrt{N\log N}}\enspace. 
\end{equation}
To bound $\Ex{Z_N}$, it is then sufficient to prove that
$$
\Exp{V}{Z_N'}\leq \sqrt{(2-\alpha)\vartheta_UN \log N}+o(\sqrt{N \log N})\enspace.
$$
Let $\epsilon>0$ and $I_N^\epsilon=\set{i\telque V_i\ge 1-N^{-1/2+\epsilon}}$. By Jensen's inequality and the argument of Pisier \cite{Pisier:1983}, for any $\lambda>0$,

 \begin{align*}
  \Exp{V}{Z_N'}&\le \frac1{\lambda}\log\paren{\paren{\sum_{i\in I_N^\epsilon}+\sum_{i\notin I_N^\epsilon}}\Exp{V}{e^{\lambda\paren{ S_i-\sum_{k=1}^{N}\frac{1}{1+V_k}}}}}\enspace.
\end{align*}
Let $\Sc=\sum_{k=1}^{N}X_k$ where, given $\mathbb{V}_1^N$, the $X_k$ are independent Bernoulli variables with respective parameters $1/(1+V_k)$ and $\Sc'=\sum_{k=1}^{N}Y_k$, where the $Y_k$ are independent Bernoulli variables with respective parameters 
 \[\frac{1-N^{-1/2+\epsilon}}{1-N^{-1/2+\epsilon}+V_k}\enspace.\]
 The variable $\Sc$ represents the score obtained by a player with strength $1$ playing against all the others, so it clearly dominates stochastically each $S_i$. Likewise, $\Sc'$ represents the score obtained by a player with strength $1-N^{-1/2+\epsilon}$ playing against all the others, so it dominates stochastically each $S_i$, with $i\notin I_N^\epsilon$. Therefore,
 \begin{multline}\label{majZ}
  \Exp{V}{Z_N'} \le \\
   \frac1{\lambda}\log\paren{|I_N^\epsilon|\Exp{V}{e^{\lambda\paren{\Sc-\Exp{V}{\Sc}}}}+|(I_N^\epsilon)^c|\Exp{V}{e^{\lambda\paren{\Sc'-\sum_{k=1}^{N}\frac{1}{1+V_k}}}}}.
 \end{multline}
 Let $\lambda_N=C\sqrt{\log N/N}$ where $C$ is a constant that will be defined later. For any $1\leq k\leq N$,
 \[\Exp{V}{Y_k}-\frac1{1+V_k}\le -\frac{N^{\epsilon}}{\sqrt N}\frac{V_k}{(1+V_k)^2}\]
 which gives following the upper bound:
 \begin{align}
&\notag\Exp{V}{e^{\lambda_N\paren{\Sc'-\sum_{k=1}^{N}\frac{1}{1+V_k}}}}\le \prod_{k=1}^N\Exp{V}{e^{\lambda_N\paren{Y_k-\Exp{V}{Y_k}}}}e^{-\lambda_N \frac{N^{\epsilon}}{\sqrt N}\frac{V_k}{(1+V_k)^2}}.
\end{align}
By \lemref{LapBernoulli}, for $N$ large enough,
\begin{align*}
\Exp{V}{e^{\lambda_N\paren{\Sc'-\sum_{k=1}^{N}\frac{1}{1+V_k}}}}&\le \prod_{k=1}^Ne^{\frac{\lambda_N^2}2-\lambda_N \frac{N^{\epsilon}}{\sqrt N}\frac{V_k}{(1+V_k)^2}}\\
&=  e^{-CN^{\epsilon}\sqrt{\log N}\paren{\frac1N\sum_{k=1}^N\frac{U_k}{(1+U_k)^2}-C\frac{\sqrt{\log N}}{2N^{\epsilon}}}}\enspace.
\end{align*}
As the strong law of large numbers shows that, $\P$-almost surely,
$$
\lim_{N\to\infty} \frac1N\sum_{k=1}^N\frac{U_k}{(1+U_k)^2}-C\frac{\sqrt{\log N}}{2N^{\epsilon}}=\vartheta_U>0
$$
 we obtain
\begin{align}
\label{eq:MajTermeReste1}\Exp{V}{e^{\lambda_N\paren{\Sc'-\sum_{k=1}^{N}\frac{1}{1+V_k}}}}=o\left( e^{-N^{\epsilon}}\right),\quad \P-\text{a.s.}\enspace.
\end{align}
We turn now to the other term in the right hand side of \eqref{majZ}: using \lemref{LapBernoulli} and the law of the iterated-logarithm, $\P$-almost-surely,
\begin{align}
\notag \Exp{V}{e^{\lambda_N\paren{\Sc-\sum_{k=1}^{N}\frac{1}{1+V_k}}}}&=\prod_{k=1}^N\Exp{V}{e^{\lambda_N\paren{X_k-\Ex{X_k}}}}\\
\notag &\le e^{\frac{\lambda_N^2}2\sum_{k=1}^N\Var\paren{X_k}+O\paren{\frac{\log^{3/2} N}{\sqrt{N}}}}\\
\label{eq:MajTermePrinc1} &\le e^{\frac{N\lambda_N^2}2\vartheta_U+O\paren{\frac{\log^{3/2} N}{\sqrt{N}}}}\quad \enspace.
\end{align}
It remains to control $|I_N^\epsilon|$. By \eqref{eq:AssumptionA}, $\P$-almost surely,
\[\Pr{U>1-1/N^{1/2-\epsilon}}=N^{-\alpha/2+\epsilon\alpha}e^{o(\log N)}\enspace,\] 
then it is easy to prove, applying Borel-Cantelli lemma, that   
\begin{equation}\label{eq:SizeINeps}
|I_N^\epsilon|= N^{1-\alpha/2+\epsilon\alpha}e^{o(\log N)},\quad \P-\text{almost surely}\enspace. 
\end{equation}

Therefore, \eqref{eq:MajTermeReste1}, \eqref{eq:MajTermePrinc1} and \eqref{eq:SizeINeps} prove that, $\P$-almost-surely
\begin{align*}
  \Exp{V}{Z_N'}
 & \le(1-\alpha/2+\alpha\epsilon)\frac{\log N}{\lambda_N}+\frac{N\lambda_N}2\vartheta_U+o\paren{\frac{\log N}{\lambda_N}}\enspace.
 \end{align*}
 Hence, choosing $C=\sqrt{\frac{(2-\alpha+2\alpha\epsilon)}{\vartheta_U}}$ that is $\lambda_N=\sqrt{\frac{(2-\alpha+2\alpha\epsilon)}{\vartheta_U}\frac{\log N}N}$ , we get 
 \begin{align*}
  \Exp{V}{Z_N'}&\le\sqrt{(2-\alpha+2\alpha\epsilon)\vartheta_UN\log N}+o(\sqrt{N\log N})\quad \P\text{-a.s.}\enspace.
 \end{align*}
As the result holds for any $\epsilon>0$ small enough, this gives the upper bound on $\Exp{V}{Z_N}$.

%
   
Proceeding as in the proofs of \eqref{majZ} and \eqref{eq:MajTermePrinc1}, but choosing now $\lambda_N=\sqrt{\frac{\gamma_0\log N}{N\vartheta_U}}$, we get the upper bound for $\Exp{V}{\max_{i\in  G_{N^{\gamma_0}}}S_i}$.

Applying \eqref{eq:SizeINeps} with $\epsilon=\nu$, we get that, $\P$-almost surely, for $N$ large enough, 
\[|G_{N^{\gamma_1}}|=N^{\gamma_1}=N^{1-\alpha/2+2\alpha\nu}>N^{1-\alpha/2+\alpha\nu}e^{o(\log N)}=|I_N^\nu|\enspace.\]
Therefore, for any $i\notin G_{N^{\gamma_1}}$, $V_i\le 1-1/N^{1/2-\nu}$.
%
We can prove as in the other cases that $\P$-almost surely, 
$$\Exp{V}{\max_{i\notin G_{N^{\gamma_1}}}S_i-\sum_{k=1}^{N}\frac{1-1/N^{1/2-\nu}}{1-1/N^{1/2-\nu}+V_k}}=O(\sqrt{N\log N})\enspace.$$
It remains to remark that, $\P$-almost surely, by \eqref{eq:LogItere} and  the strong law of large numbers,
\begin{align*}
&\sum_{k=1}^{N}\frac{1-1/N^{1/2-\nu}}{1-1/N^{1/2-\nu}+V_k}=
\sum_{k=1}^{N}\frac1{1+V_k}-1/N^{1/2-\nu}\sum_{k=1}^{N}\frac{V_k}{(1+V_k)^2}+o\paren{N^{1/2+\nu}}\\
&=N\Ex{\frac1{1+U}}-N^{1/2+\nu}\vartheta_U+o\paren{N^{1/2+\nu}}\enspace.
\end{align*}
This concludes the proof of the upper bound on $\Exp{V}{\max_{i\notin  G_{N^{\gamma_1}}}S_i}$.

\medskip

\emph{Lower bound on $\Exp{V}{Z_N}$.}
Let us start with the following lemma which says that $Z_N$ stochastically dominates the maximum of independent
copies of the variables $S_i$.

\begin{lemma}\label{lem:tech}
For any $a>0$ we have, $\P$-almost surely, 
$$\Prp{V}{Z_N\le a}\le \prod_{i=1}^N\Prp{V}{S_i\le a}\enspace.$$
\end{lemma}

\begin{proof} We proceed by induction, we provide a detailed proof of the first step, the other ones follow the same lines.
Let $\tilde{X}_{2,1}$ denote a copy of $X_{2,1}$, independent of $(X_{i,j})_{1\le i<j\le N}$ and let $S^1_2=\tilde{X}_{2,1}+\sum_{i= 3}^NX_{2,i}$, $M_2=S_1\vee S_2$, $\tilde{M}_2=S_1\vee S^1_2$, $A=\set{\max_{i\ge 3}S_i\le a}$. Write 
$\set{Z_N\le a}=\set{M_2\le a}\cap A$. Simple computations show that:
\begin{align*}
\Prp{V}{Z_N\le a}=&\Prp{V}{\set{\tilde{Z}_N\le a}}-\Prp{V}{\set{\tilde{M}_2= a}\cap\set{M_2=\tilde{M}_2+1}\cap A}\\
 &\quad+\Prp{V}{\set{\tilde{M}_2= a+1}\cap\set{M_2=\tilde{M}_2-1}\cap A}\enspace.
\end{align*}
In addition, for $x\in\{0,1\}$,
\begin{align*}
\set{M_2-\tilde{M}_2= 1-2x} &= \set{X_{1,2}=\tilde{X}_{2,1}=x\ ,\ \sum_{i=3}^NX_{2,i}\ge x+\sum_{i=3}^NX_{1,i}}\enspace.
\end{align*}
Now, recall that $X_{1,2}$ and $\tilde{X}_{1,2}$ are independent of $\sum_{i=3}^NX_{2,i},\sum_{i=3}^NX_{1,i}$ and $A$, then, for $x\in\{0,1\}$,
\begin{multline*}
\Prp{V}{\set{\tilde{M}_2= a+x}\cap\set{M_2=\tilde{M}_2+1-2x}\cap A}\\
=\Prp{V}{X_{1,2}=\tilde{X}_{2,1}=x}\Prp{V}{\set{\sum_{i=3}^NX_{2,i}=a,\sum_{i=3}^NX_{1,i}\le a-x}\cap A} 
\end{multline*}

As $\Prp{V}{X_{1,2}=\tilde{X}_{2,1}=0}=\Prp{V}{X_{1,2}=\tilde{X}_{2,1}=1}$ 
we obtain
\[\Prp{V}{Z_N\le a}\le \Prp{V}{\tilde{Z}_N\le a}\enspace.\]
\end{proof}
%
%
Let $I^0_N=\set{i\telque V_i\ge 1-N^{-1/2}}$ and $S=\sum_{i=2}^{N}X_i$ where the $X_i$ are independent and $X_i\sim\Bcal(1/(1+V_i/(1-N^{-1/2}))$. The variable $S$ is stochastically dominated by any $S_i$ with $i\in I^0_N$. It follows from \lemref{tech} that, 
\[\Prp{V}{Z_N< a}\le \prod_{i=1}^{N}\Prp{V}{S_i< a}\le \prod_{i\in I^0_N}\Prp{V}{S_i< a}\le \Prp{V}{S< a}^{|I^0_N|}\enspace.\]
For any $\epsilon\in(0,1-\alpha/2)$, denote $\gamma_N=\sqrt{(2-\alpha-2\epsilon)\vartheta_UN\log N}$. The previous inequality yields 
\begin{align*}
 \Exp{V}{Z_N}&\geq(\gamma_N+\Exp{V}{S})\Prp{V}{Z_N-\Exp{V}{S}\geq\gamma_N}\\
&\geq (\gamma_N+\Exp{V}{S})(1-\Prp{V}{S-\Exp{V}{S}<\gamma_N}^{|I^0_N|})\enspace.
 \end{align*}
Denote $\lambda_N=\sqrt{\frac{\log N}N}$. By \lemref{LapBernoulli}, for any $u\in \R_+$ and any $N$ such that $u\lambda_N\le 1$,
\begin{align*}
 \log\Exp{V}{e^{u\lambda_N(S-\Ex{S})}}
 &=\sum_{i=1}^{N-1}\paren{\frac{u^2\lambda_N^2}2\frac{V_i}{(1+V_i)^2}+O(\lambda_N^3)}\\
 &=\frac{u^2}2\vartheta_U\log N+O\paren{\frac{(\log N)^{3/2}}{\sqrt{N}}}\ \P\mbox{-a.s.}\enspace.
\end{align*}
The last line is obtained thanks to the law of iterated logarithm. Hence,
\[\lim_{N\to\infty}\frac1{\log N}\log\Exp{V}{e^{u\log N\frac{S-\Exp{V}{S}}{\sqrt{N\log N}}}}= \frac{u^2}2\vartheta_U\enspace.\]
The same argument applied on the variables $-X_i$ shows that the previous inequality actually holds for any $u\in \R$.
Therefore, using \thmref{GE} in Appendix with the sequence of random variables $\zeta_N=\frac{S-\Exp{V}{S}}{\sqrt{N\log N}}$,
\[\liminf_{N}\frac1{\log N}\log\Prp{V}{S-\Exp{V}{S}>\gamma_N}\ge -1+\alpha/2+\epsilon\enspace.\]
In particular, since 
$\log |I_N^0|\sim (1-\alpha/2)\log N$, for $N$ large enough
\begin{align*}
\Prp{V}{S-\Exp{V}{S}<\gamma_N}^{|I^0_N|}
&\le \paren{1-N^{-1+\alpha/2+\epsilon/2}}^{|I_N^0|}\\
&\le e^{-N^{\epsilon/4}} 
\enspace.
\end{align*}
Since, by the law of iterated logarithm,
\begin{align*}
\Exp{V}{S}&=\sum_{i=1}^N\frac{1-1/\sqrt{N}}{1-1/\sqrt{N}+U_i}-\frac{1}{1+V_1}\\
&\ge N\Ex{\frac1{1+U}}+o(\sqrt{N\log N})\enspace, 
\end{align*}
we  obtain that, for any $\epsilon>0$,
\[\Exp{V}{Z_N}\geq N\Ex{\frac1{1+U}}+\sqrt{(2-\alpha-2\epsilon)\vartheta_UN\log N}+o\paren{\sqrt{N\log N}}\]
which concludes the proof.
\end{proof}
%
\subsection{Conclusion of the proof of \thmref{Cutoff}}
%
Choosing $u=\log\log N$ in \eqref{eq:lowerboundS}, a slight extension of the law of iterated logarithm gives that,  $\P$-almost surely, with $\P_V$-probability going to $1$, 
\begin{align*}
S_{N+1}&\ge \sum_{i=1}^N\frac{v_{N+1}}{V_i+v_{N+1}}-\sqrt{\frac{Nu_N}2}=N\Ex{\frac{v_{N+1}}{U+v_{N+1}}}+o\paren{\sqrt{N\log N}}
\end{align*}
and
\[S_{N+1}\le \sum_{i=1}^N\frac{v_{N+1}}{V_i+v_{N+1}}+\sqrt{\frac{Nu_N}2}=N\Ex{\frac{v_{N+1}}{U+v_{N+1}}}+o\paren{\sqrt{N\log N}}\enspace.\]
Therefore, there exists $\epsilon^1_N\to0$ such that the first statement of \eqref{eq:ConSZ} holds $\P$-almost surely with 
\[s_{\pm}^N=N\Ex{\frac{v_{N+1}}{U+v_{N+1}}}\pm\epsilon^1_N\sqrt{N\log N}\enspace.\]
By \lemref{ConcMax} 
with $u=\log\log N$, and \lemref{AsymptoticOfEZ_N},  $\P$-almost surely, with $\P_V$-probability going to $1$,
\begin{align*}
Z_N
&= N\Ex{\frac{1}{1+U}}+\sqrt{(2-\alpha)\vartheta_U N \log N}+o\paren{\sqrt{N\log N}}.
\end{align*}
Therefore, there exists $\epsilon_N^2\to0$ such that the second statement of \eqref{eq:ConSZ} holds with 
\begin{align*}
z_+^N+1&=N\Ex{\frac{1}{U+1}}+\sqrt{(2-\alpha)\vartheta_UN \log N}+\epsilon_N^2\sqrt{N\log N}\enspace,\\
z_-^N&=N\Ex{\frac{1}{U+1}}+\sqrt{(2-\alpha)\vartheta_UN \log N}-\epsilon_N^2\sqrt{N\log N}\enspace. 
\end{align*}
Moreover,
\[
\Ex{\frac{v_{N+1}}{U+v_{N+1}}-\frac{1}{1+U}}=(v_{N+1}-1)\Ex{\frac{U}{(U+1)(U+v_{N+1})}}\enspace.\]
Hence, denoting $\epsilon^3_N=\epsilon_N^1+\epsilon_N^2$, the inequality $s_-^N>z_+^N+1$ is verified if 
\begin{align*}
(v_{N+1}-1)\Ex{\frac{U}{(U+1)(U+v_{N+1})}}&\ge \paren{\sqrt{(2-\alpha)\vartheta_U}+\epsilon^3_N}\sqrt{\frac{\log N}N}\\
\end{align*} 
that is if
\begin{align}\label{minor}
\frac{v_{N+1}-1}{\epsilon_N}\frac{\Ex{\frac{U}{(U+1)(U+v_{N+1})}}}{\vartheta_U}&\ge 1+\frac{\epsilon^3_N}{\sqrt{(2-\alpha)\vartheta_U}}\enspace.
\end{align} 
where $\epsilon_N=\sqrt{2-\alpha}\;\vartheta_U^{-1/2}\sqrt{\frac{\log N}{N}}$ is the value appearing in the statement of \thmref{Cutoff}.
We now prove by contradiction that 
\begin{align}\label{minor2}
\liminf_{N\to\infty}\frac{v_{N+1}-1}{\epsilon_N}\frac{\Ex{\frac{U}{(U+1)(U+v_{N+1})}}}{\vartheta_U}>1\enspace.
\end{align} 
Suppose  there is a subsequence of $v_{N+1}$ (that we still call $v_{N+1}$) such that \eqref{minor2} is not true. As $\liminf(v_{N+1}-1)/\epsilon_N>1$, it means that for $N$ sufficiently large, $v_{N+1}\ge 1+\delta$ for some $\delta>0$. But in this case, the LHS of \eqref{minor2} clearly goes to infinity as $N\to\infty$. That contradicts our initial assumption and then \eqref{minor} is verified for $N$ large enough.

The proof that $s_+^N<z_-^N$ when $\liminf(v_{N+1}-1)/\epsilon_N<1$ follows the same arguments.
\subsection{Conclusion of the proof of \thmref{SizeSetWinners}}
By  Lemma \ref{lem:ConcMax} with $u=\log\log N$ and 
\lemref{AsymptoticOfEZ_N},
 $\P$-almost surely, with $\P_V$-probability going to $1$,
\begin{align*}
\max_{i\in G_{N^{\gamma_0}}}S_i &\leq N\Ex{\frac{1}{1+U}}+\sqrt{2\gamma_0\vartheta_UN \log N}+o\paren{\sqrt{N\log N}} \enspace. 
\end{align*}
Since $\gamma_0<1-\alpha/2$, the first item of \eqref{eq:UpperBoundsinG} holds for $N$ large enough with 
\[z_{0}^N=N\Ex{\frac{1}{1+U}}+\sqrt{\paren{\gamma_0+1-\frac\alpha2}\vartheta_UN \log N}\enspace.\]
It is clear that $z_0^N<z_-^N$ for $N$ large enough since, by definition $\gamma_0+1-\frac\alpha2<2-\alpha$.
By 
\lemref{ConcMax}
with $u=\log\log N$ and \lemref{AsymptoticOfEZ_N}, $\P$-almost surely, with $\P_V$-probability going to $1$,
\begin{align*}
\max_{i\notin G_{N^{\gamma_1}}}S_i&\leq N\Ex{\frac1{1+U}}-N^{1/2+\nu}\vartheta_U+o\paren{N^{1/2+\nu}}\enspace,
\end{align*}
where $\nu=\frac{\gamma_1-(1-\alpha/2)}{2\alpha}>0$. Hence, the second item of \eqref{eq:UpperBoundsinG} holds for $N$ large enough with 
\[z_{1}^N=N\Ex{\frac{1}{1+U}}\enspace,\]
which is clearly smaller than $z_-^N$.
\appendix
\section*{\large Appendix}
To evaluate the various expectations of suprema of random variables, the following result is used repeatedly. Its proof is straightforward and therefore omitted.
\begin{lemma}\label{lem:LapBernoulli}
 Let $X$ be a Bernoulli distribution with parameter $p\in[0,1]$ and $a\in[0, 1]$, then
 \[1+\frac{a^2}2p(1-p)\le \Ex{e^{a(X-\Ex{X})}}\le e^{p(1-p)a^2\paren{\frac{1}2+\frac{4 e^2}3a}}\enspace.\]
\end{lemma}
\medskip

We also recall for reading convenience two well-known results. The first one is the bounded difference inequality (see Theorem 5.1 in \cite{Massart:2007}):

\begin{theorem}
\label{thm:bdi}
	Let $\Xcal^n=\Xcal_1\times\dots\times\Xcal_n$ be some product measurable space and $\Psi:\Xcal^n\to\R$ be some measurable functional satisfying the bounded difference condition:
	$$
	\absj{\Psi(x_1,\dots,x_i,\dots,x_n)-\Psi(x_1,\dots,y_i,\dots,x_n)}\leq1
	$$
	for all $x\in\Xcal^n,\ y\in\Xcal^n$, $i\in\set{1,\dots,n}$. Then the random variable $Z=\Psi(X_1,\dots,X_n)$ satisfies for any $u>0$,
	$$
	\Pr{Z\geq\Ex{Z}+\sqrt{\frac n2u}}\leq e^{-u}\quad \text{ and }\quad \Pr{Z\leq\Ex{Z}-\sqrt{\frac n2u}}\leq e^{-u}\enspace.
	$$
\end{theorem}

\medskip

The second result is a simple consequence of G\"artner-Ellis theorem (see Theorem 2.3.6 in \cite{Dembo_Zeitouni:2002}) and of the Fenchel-Legendre transform of a centered Gaussian distribution.
\begin{theorem}\label{thm:GE}
Consider a sequence of r.v. $\left(\zeta_n\right)_{n\in\N}$ and a deterministic sequence $(a_n)_{n\in\N}\to\infty$ such that for any $u\in\R$, 
$$\lim_{n\to\infty}\frac{1}{a_n}\log\Ex{e^{ua_n\zeta_n}}=\frac{u^2\sigma^2}2\enspace.
$$
Then, for any $x>0$,
$$
\liminf_{n\to\infty}\frac{1}{a_n}\log\Pr{\zeta_n> x}\geq -\frac{x^2}{2\sigma^2}\enspace.
$$	
\end{theorem}

 \bibliographystyle{plain}
 \bibliography{BT}

\end{document}